\documentclass[reqno,12pt]{amsart}

\usepackage{fullpage}
\usepackage[english]{babel}
\usepackage[T1]{fontenc}
\usepackage{graphicx}
\usepackage{amsmath}
\usepackage{amsfonts}
\usepackage{amssymb}
\usepackage{a4wide}

\newcommand\qbi[3]{{{#1}\atopwithdelims[]{#2}}_{#3}}
\newcommand\bi[2]{{{#1}\atopwithdelims(){#2}}}

\newtheorem{thm}{Theorem}[section]

\newtheorem{defn}[thm]{Definition}

\newtheorem{prop}[thm]{Proposition}
\newtheorem{lem}[thm]{Lemma}

\numberwithin{equation}{section}

\author{B. Adamczewski, J. P. Bell, \'E. Delaygue and F. Jouhet}
\title{Congruences modulo cyclotomic polynomials and algebraic independence for $q$-series}
\date{}

\thanks{This project has received funding from the European Research Council (ERC) under the European Union's Horizon 2020 research and innovation programme under the Grant Agreement No 648132.}

\begin{document}
\maketitle

\begin{abstract}We prove congruence relations modulo cyclotomic polynomials for multisums of 
$q$-factorial ratios, therefore generalizing many well-known $p$-Lucas congruences. 
Such congruences connect  various classical 
 generating series  to their $q$-analogs. Using this, we prove a propagation phenomenon: 
when these generating series are algebraically independent, this is also the case for their $q$-analogs.
\end{abstract}

\section{Introduction and main results}\label{sec: intro}

After the seminal work of Lucas~\cite{Lu}, a great attention has been paid on congruences modulo prime numbers $p$ satisfied by various combinatorial sequences related to binomial coefficients.  
A typical example of these so-called $p$-Lucas congruences is given by:
\begin{equation}\label{centralLucasCongruence}
\bi{2(m+np)}{m+np}^r\equiv \bi{2m}{m}^r\bi{2n}{n}^r\mod p,
\end{equation}
where $0\leq m\leq p-1$, $r\geq 1$, and $n\geq 0$.  
In terms of generating series these congruences~\eqref{centralLucasCongruence} 
translate as 
\begin{equation}\label{2nnfonction}
g_r(x)\equiv A(x)g_r(x^p)\mod p\mathbb{Z}[[x]],
\end{equation}
where $$g_r(x):= \sum_{n= 0}^{\infty} {2n \choose n}^r x^n$$ and $A(x)$ is a polynomial (depending on $r$ and $p$) in $\mathbb Z[x]$ of degree at most $p-1$. 
This functional point of view led the authors of~\cite{ABD1} to define general sets of multivariate 
power series including the following one which is of particular interest for our purpose. 
  
\begin{defn}\label{def: ldrs} Let $d$ be a positive integer and ${\bf x} = (x_1,\ldots,x_d)$ 
be a vector of indeterminates.  
We let $\mathfrak{L}_d$ denote the set of all power series 
$g({\bf x})$ in $\mathbb{Z}[[{\bf x}]]$ with constant term equal to $1$ and such that for 
every prime number $p$: 

{\rm(i)}  there exist a positive integer $k$ 
and a polynomial  $A$ in $\mathbb{Z}[{\bf x}]$ 
satisfying
\begin{equation*}\label{eq cond}
g(\mathbf{x})\equiv A(\mathbf{x})g\big(\mathbf{x}^{p^{k}}\big)\mod  p\mathbb{Z}[[\mathbf{x}]].
\end{equation*}

{\rm(ii)}   $\deg_{x_i}(A)\leq p^{k}-1$ for all $i$, $1\leq i\leq d$. 

\end{defn}
Using $p$-adic computations inspired by works of Christol and Dwork, it was proved in~\cite{ABD1} that a large family of multivariate generalized hypergeometric series belongs to 
$\mathfrak L_d$. This provides, by specialization, a unified way 
to reprove most of known $p$-Lucas congruences as well as to find many new ones.  
In addition, a general method to prove algebraic independence of power series whose coefficients satisfy $p$-Lucas type congruences was developed.  
 Let us illustrate this approach with the following example. 
 In 1980, Stanley~\cite{St80} conjectured that the series $g_r$
are transcendental over $\mathbb C(x)$ except for $r=1$, in which case we have 
$g_1(x)=(\sqrt{ 1 - 4x})^{-1}$.
 He also proved the transcendence for $r$ even. The conjecture was solved independently by Flajolet~\cite{Fl} through asymptotic considerations and by Sharif and Woodcock~\cite{SW2} by using the previously mentioned Lucas congruences. Incidentally, this result is also a consequence of the interlacing criterion proved by Beukers and Heckman~\cite{BH} for generalized hypergeometric series. 
Though there are three different ways to obtain the transcendence of $g_r$ for $r\geq2$,  not much was apparently known about their algebraic independence, until Congruence \eqref{centralLucasCongruence} was used in~\cite{AB13,ABD1} to prove the following result: 
{\it all elements of the set $\left\{g_r(x)\,:\, r\geq 2 \right\}$ are algebraically independent 
over $\mathbb C(x)$}.

In the present work, we aim at generalizing the approach of~\cite{ABD1} in the setting of 
$q$-series. It started with the following observation which can be derived from \cite{Fr,Sa,St}: 
\begin{equation}\label{qLucasCongruence}
\qbi{2(m+nb)}{m+nb}{q}^r\equiv \qbi{2m}{m}{q}^r\binom{2n}{n}^r\mod \phi_b(q)\mathbb{Z}[q] \,,
\end{equation}
where $n,m,b,r$ are nonnegative integers with $b\geq 1$, $0\leq m\leq b-1$, and $\phi_b(x)$ denotes the $b$th cyclotomic polynomial over $\mathbb Q$. Here,  
for every complex number $q$, the central $q$-binomial coefficients are defined as
$$
\qbi{2n}{n}{q}:=\frac{[2n]_q!}{[n]_q!^2}\in\mathbb{N}[q],\;\mbox{where}\;[n]_q!
:=\prod_{i=1}^n\frac{1-q^i}{1-q}
$$
is the $q$-analog of $n!$. It is implicitly considered as a polynomial in $q$ so that the formula is still valid for $q=1$. In particular, one has $[n]_1!=n!$ and 
the congruence \eqref{qLucasCongruence} allows one to recover \eqref{centralLucasCongruence} 
since $\phi_p(1)=p$.  
Again in terms of generating series, \eqref{qLucasCongruence}  
 translates as
\begin{equation}\label{transfert2nn}
f_r(q;x)\equiv A(q;x)g_r(x^b)\mod \phi_b(q)\mathbb{Z}[q]\,[[x]]\,,
\end{equation}
where $A(q;x)$ is a polynomial in $\mathbb{Z}[q][x]$ of degree (in $x$) at most $b-1$ 
and we have set
$$
f_r(q;x):=\sum_{n=0}^\infty\qbi{2n}{n}{q}^rx^n.
$$
This provides an arithmetic connection between the generating series $g_r(x)$ and 
its $q$-analog $f_r(q;x)$. 
It leads us to associate a set $\mathcal{D}(q;g)$ of 
$q$-deformations with every element $g$  in $\mathfrak{L}_d$. 
We stress that $\mathcal{D}(q;g)$ is closed under $q$-derivation. 

\begin{defn}\label{def: lastdrsb} {\em Let $q$ be a fixed nonzero complex number.   
Let $g(\mathbf{x})$ be a power series in $\mathfrak{L}_d$. 
We let $\mathcal{D}(q;g)$ denote the set of all nonzero power series 
$f(q;{\bf x})$ in $\mathbb{Z}[q][[{\bf x}]]$  such that for all integers $b\geq 1$ 
there exists a polynomial $A(q;{\bf x})$ with coefficients in $\mathbb{Z}[q]$  satisfying: 
$$f(q;\mathbf{x})\equiv A(q;{\bf x}) g\big(\mathbf{x}^{b}\big) 
\mod \phi_b(q)\mathbb{Z}[q][[\mathbf{x}]].
$$}
\end{defn}

Our first result shows a propagation phenomenon of algebraic independence  from generating 
series in $\mathfrak L_d$ to their $q$-analogs. 

\begin{thm}\label{thm: propagation} Let $q$ be a nonzero complex number. 
Let $g_1(\mathbf{x}),\dots,g_n(\mathbf{x})$ be power series in $\mathfrak{L}_d$, which are algebraically independent over $\mathbb{C}(\mathbf{x})$. Then for any $f_i(q;\mathbf{x})$ in $\mathcal{D}(q;g_i)$,  $1\leq i\leq  n$, the series $f_1(q;\mathbf{x}),\dots,f_n(q;\mathbf{x})$ are also 
algebraically independent over $\mathbb{C}(\mathbf{x})$.
\end{thm}

 
It immediately implies that all elements of the set $\left\{f_r(q;x)\,:\, r\geq 2 \right\}$ are algebraically independent over $\mathbb C(x)$ for all nonzero complex numbers $q$. More generally, there is a long  tradition for combinatorists in studying  
$q$-analogs of famous sequences of natural numbers, as the additional variable $q$ gives the opportunity to refine the enumeration of combinatorial objects counted by the $q=1$ case. To some extent, the nature of a generating series reflects the underlying structure of the objects it counts~\cite{BM06}.  By nature, we mean for instance whether the generating series is  rational, algebraic, or $D$-finite. In the same line,  algebraic (in)dependence of generating series can be considered as a reasonable way to  measure how distinct families of combinatorial objects may be (un)related.
\medskip

It is known from~\cite{ABD1} that many generating series $g$ of multisums of factorial ratios belong to $\mathfrak{L}_d$. For such series $g$, we will define $q$-analogs and prove that they lie in the set 
$\mathcal{D}(q;g)$. This will yield at once algebraic independence results, but also many generalizations of Congruence \eqref{qLucasCongruence}. Finding congruences with respect to cyclotomic polynomials is actually not a recent problem 
(see for instance~\cite{Sa, GZ, Pa} and the references cited there).

Our second result below is a general congruence relation extending~\eqref{qLucasCongruence} to the multidimensional case, by considering  $q$-factorial ratios in the spirit of the ones in~\cite{WZ}. For  positive integers $d,u,v$, let $e=(\mathbf{e}_1,\dots,\mathbf{e}_u)$ and $f=(\mathbf{f}_1,\dots,\mathbf{f}_v)$ be tuples of vectors in $\mathbb{N}^d$. For $\mathbf{n}\in\mathbb{N}^d$, we 
define a $q$-analog of multidimensional factorial ratios (see Section~\ref{sec:notation} for precise notations) by:
$$
\mathcal{Q}_{e,f}(q;\mathbf{n}):=\frac{[\mathbf{e}_1\cdot\mathbf{n}]_q!\cdots[\mathbf{e}_u\cdot\mathbf{n}]_q!}{[\mathbf{f}_1\cdot\mathbf{n}]_q!\cdots[\mathbf{f}_v\cdot\mathbf{n}]_q!} \, \cdot
$$
Furthermore, we consider the Landau step function $\Delta_{e,f}$ defined on $\mathbb{R}^d$ by 
$$
\Delta_{e,f}(\mathbf{x}):=\sum_{i=1}^u\lfloor\mathbf{e}_i\cdot\mathbf{x}\rfloor-\sum_{j=1}^v\lfloor\mathbf{f}_j\cdot\mathbf{x}\rfloor.
$$
We also define $|e|=\sum_{i=1}^u\mathbf{e}_i$, $|f|=\sum_{j=1}^v\mathbf{f}_j$, and  set: 
$$
\mathcal{D}_{e,f}:=\big\{\mathbf{x}\in[0,1)^d\,:\,\textup{there is $\mathbf{t}$ in $e$ or $f$ such that $\mathbf{t}\cdot\mathbf{x}\geq 1$}\big\}.
$$
\begin{prop}\label{prop: congq}
Let $e$ and $f$ be two tuples of vectors in $\mathbb{N}^d$ such that $|e|=|f|$ and $\Delta_{e,f}$ is greater than or equal to $1$ on $\mathcal{D}_{e,f}$. Then, for every positive integer $b$, every $\mathbf{a}$ in $\{0,\dots,b-1\}^d$ and every $\mathbf{n}$ in $\mathbb{N}^d$, we have $\mathcal{Q}_{e,f}(q;\mathbf{n})\in\mathbb{Z}[q]$ and
$$
\mathcal{Q}_{e,f}(q;\mathbf{a}+\mathbf{n}b)\equiv\mathcal{Q}_{e,f}(q;\mathbf{a})\mathcal{Q}_{e,f}(1;\mathbf{n})\mod\phi_b(q)\mathbb{Z}[q].
$$ 
\end{prop}
Proposition~\ref{prop: congq} extends many known results, both for $q$-analogs and $p$-Lucas congruences. For instance, choosing $d=1$, $u=1$, $v=2$, $e_1=2$, and $f_1=f_2=1$  yields~\eqref{qLucasCongruence}, while taking $b$ prime and $q=1$ allows one to recover Proposition~8.7  in \cite{ABD1}. As we will see in Section~\ref{sec:qcongruences}, Proposition~\ref{prop: congq} also leads to congruences for (multi-)sums of $q$-factorial ratios. 
As an illustration, we provide below two examples connected to the famous Ap\'ery sequences.

\begin{prop}\label{prop:qApery}
Consider for a given nonnegative integer $t$ the following $q$-analogs of the Ap\'ery sequences
$$
a_n(q):=\sum_{k=0}^nq^{t k}\qbi{n}{k}{q}^2\qbi{n+k}{k}{q}\quad\textup{and}\quad b_n(q):=\sum_{k=0}^nq^{t k}\qbi{n}{k}{q}^2\qbi{n+k}{k}{q}^2.
$$ 
Then, for all nonnegative integers $n,m,b$ with $b\geq 1$, $0\leq m\leq b-1$, we have 
$$
a_{m+nb}(q)\equiv a_{m}(q)a_{n}(1)\mod \phi_b(q)\mathbb{Z}[q]\quad\textup{and}\quad b_{m+nb}(q)\equiv b_{m}(q)b_{n}(1)\mod \phi_b(q)\mathbb{Z}[q].$$
\end{prop}

Setting 
$$
F_{e,f}(q;\mathbf{x}):=\sum_{\mathbf{n}\in\mathbb{N}^d}\mathcal{Q}_{e,f}(q;\mathbf{n})\mathbf{x}^{\mathbf{n}}
$$
and assuming the conditions of Proposition~\ref{prop: congq}, we obtain that  
$F_{e,f}(q;\mathbf{x})$ belongs to $\mathcal{D}(q;F_{e,f}(1;\mathbf{x}))$, as 
 $F_{e,f}(1;\mathbf{x})$ is in $\mathfrak{L}_d$ by \cite[Proposition~8.7]{ABD1}. Theorem~\ref{thm: propagation} therefore implies that algebraic independence among series $F_{e,f}(1;\mathbf{x})$ propagates to their corresponding $q$-analogs $F_{e,f}(q;\mathbf{x})$. As noticed before, this 
 holds for the series $g_r(x)$ and their $q$-analogs $f_r(q;x)$. 
Proposition~\ref{prop: congq} and a result about specializations of the series 
$F_{e,f}(q;\mathbf{x})$ (stated in Section~\ref{sec:qcongruences}) actually provide 
much more general results, such as the following one.

\begin{prop}\label{prop:indepqApery}
For a fixed nonzero complex number $q$, let $\mathcal F_q$ be the set formed by the union of the three following sets: 
$$
\left\{\sum_{n=0}^\infty \sum_{k=0}^n\qbi{n}{k}{q}^r x^n\,: \,r\geq 3\right\},\quad\left\{\sum_{n=0}^\infty
\sum_{k=0}^n\qbi{n}{k}{q}^r\qbi{n+k}{k}{q}^r x^n\,: \,r\geq 2\right\}
$$
and
$$
\left\{\sum_{n=0}^\infty \sum_{k=0}^n\qbi{n}{k}{q}^{2r}\qbi{n+k}{k}{q}^r x^n\,: \,r\geq 1\right\} \, .
$$
Then all elements of $\mathcal F_q$ are algebraically independent over $\mathbb C(x)$.
\end{prop}

Proposition~\ref{prop:indepqApery}  is derived from Proposition 1.2 in \cite{ABD1}, which corresponds to the case $q=1$. 

In the next section, we  fix some notation and recall  basic facts about Dedekind domains. In Section~\ref{sec:qcongruences}, we focus on congruence relations modulo cyclotomic polynomials and prove Proposition~\ref{prop: congq}. We also show how to derive results like Propositions~\ref{prop:qApery} and~\ref{prop:indepqApery}. Finally,  the last section is devoted to a sketch of proof of Theorem~\ref{thm: propagation}.


\section{Background and notations}\label{sec:notation}

Let us introduce some notation and basic facts that will be used throughout this extended abstract.  
Let $d$ be a positive integer. Given $d$-tuples of real numbers $\mathbf{m}=(m_1,\dots,m_d)$ and $\mathbf{n}=(n_1,\dots,n_d)$, 
 we set $\mathbf{m}+\mathbf{n}:=(m_1+n_1,\dots,m_d+n_d)$ and  
$\mathbf{m}\cdot\mathbf{n}:=m_1n_1+\cdots+m_dn_d$. If moreover  $\lambda$ is a real number, then we set 
$\lambda\mathbf{m}:=(\lambda m_1,\dots, \lambda m_d)$. 
We write $\mathbf{m}\geq\mathbf{n}$ if   we have 
$m_k\geq n_k$ for all $k$ in $\{1,\dots,d\}$. We also set $\mathbf{0}:=(0,\dots,0)$ and $\mathbf{1}:=(1,\dots,1)$. 

{\it Polynomials}. Given a   $d$-tuple of natural numbers $\mathbf{n}=(n_1,\dots,n_d)$ and a vector of 
indeterminates  $\mathbf{x}=(x_1,\dots,x_d)$, 
we will denote by ${\bf x}^{\bf n}$  the monomial $x_1^{n_1}\cdots x_d^{n_d}$.  The (total) degree 
of such a monomial is the nonnegative integer $n_1+\cdots + n_d$.  
Given a ring $R$ and a 
polynomial $P$ in $R[{\bf x}]$, we denote by $\deg P$ the (total) degree of $P$, that is 
the maximum of the total degrees of the monomials appearing in $P$ with nonzero coefficient.   
The partial degree of $P$ with respect to the indeterminate $x_i$ 
is denoted by $\deg_{x_i}(P)$. 

 {\it Algebraic power series and algebraic independence}. Let $K$ be a field.  We denote by $K[[{\bf x}]]$ the ring of formal power series with coefficients 
in $K$ and associated with the vector of indeterminates ${\bf x}$. 
We say that a power series $f({\bf x})\in K[[{\bf x}]]$ 
is algebraic if it is algebraic over the field of rational functions $K({\bf x})$, 
that is,  if there exist polynomials $A_0,\ldots, A_m$ in ${K}[{\bf x}]$, not all zero, such that $A_0({\bf x})+A_1({\bf x})f({\bf x})+\cdots+A_m({\bf x})f({\bf x})^m=0$. 
Otherwise, $f$ is said to be transcendental. 
Let $f_1,\ldots,f_n$ be in $K[[{\bf x}]]$. We say that $f_1,\ldots,f_n$ are algebraically dependent 
if they are algebraically dependent over the field $K({\bf x})$, that is, if there exists a nonzero polynomial 
$P(Y_1,\ldots,Y_n)$ in $K[{\bf x}][Y_1,\ldots,Y_n]$ 
such that $P(f_1,\ldots,f_n)=0$. 
When there is no algebraic relation between them, the power series $f_1,\ldots,f_n$ are said to be algebraically independent 
(over $K({\bf x})$). 

{\it Dedekind domains.} Let $R$ be a Dedekind domain; that is, $R$ is Noetherian, integrally closed,  
and every nonzero prime ideal of $R$ is a maximal ideal. Furthermore, 
any nonzero element of $R$ belongs to at most a finite number of maximal ideals of $R$. 
In other words, given an infinite set $\mathcal S$ of maximal ideals of $R$, then one always has 
$\bigcap_{\mathfrak p\in \mathcal S} \mathfrak p =\{0\}$. For every power series $f(\mathbf{x})=\sum_{\mathbf{n}\in\mathbb{N}^d} a(\mathbf{n})\mathbf{x}^{\mathbf{n}}$ with coefficients in $R$, we set
$$f_{|\mathfrak p}(\mathbf{x}):=\sum_{\mathbf{n}\in\mathbb{N}^d}\big(a(\mathbf{n})\bmod \mathfrak p\big)\mathbf{x}^{\mathbf{n}}\in 
(R/\mathfrak p)[[\mathbf{x}]] \, .$$ 
The power series $f_{|\mathfrak p}$ is called the reduction of $f$ modulo $\mathfrak p$. 
Let  $K$ denote the field of fractions of $R$. 
The localization of $R$ at a maximal ideal $\mathfrak p$ is denoted by $R_{\mathfrak p}$. 
Recall here that $R_{\mathfrak p}$ can be seen as the following subset of $K$: 
$$
R_{\mathfrak p}= \left\{a/b \,:\, a\in R, b\in R\setminus  \mathfrak p\right\}.
$$ 
Then $R_{\mathfrak p}$ is a discrete valuation ring and the residue field $R_{\mathfrak p}/\mathfrak p$ is equal to $R/\mathfrak p$.

\section{Some general congruences and appplications}\label{sec:qcongruences}

We first give the detailed proof of Proposition~\ref{prop: congq}, and we will then see how to derive results like Propositions~\ref{prop:qApery} and~\ref{prop:indepqApery}.

\begin{proof}[Proof of Proposition~\ref{prop: congq}]
In this proof, we write $\mathcal{Q}$ for $\mathcal{Q}_{e,f}$, $\Delta$ for $\Delta_{e,f}$ and $\mathcal{D}$ for $\mathcal{D}_{e,f}$.
Recall  that for all nonnegative integers $n$ we have 
\begin{equation*}\label{eq:qfactphi}
\frac{1-q^n}{1-q}=\prod_{b\mid n, b\geq 2}\phi_b(q)\;\Rightarrow\;[n]_q!=\prod_{b=2}^n\phi_b(q)^{\lfloor n/b\rfloor},
\end{equation*}
from which we deduce, by definition of the step function $\Delta$,  
\begin{equation}\label{eq: poly}
\mathcal{Q}(q;\mathbf{n})=\prod_{b=2}^\infty\phi_b(q)^{\Delta(\mathbf{n}/b)}.
\end{equation}
As $|e|=|f|$, the function $\Delta$ is $1$-periodic in each of its variable and one easily obtains from~\eqref{eq: poly} that $\mathcal{Q}(q;\mathbf{n})$ is in $\mathbb{Z}[q]$ for every $\mathbf{n}$ in $\mathbb{N}^d$ if, and only if $\Delta$ is nonnegative over $\mathbb{R}^d$. This proves the first part of our proposition.

  Let $x$ be a complex variable. As $|e|=|f|$, we derive
$$
\mathcal{Q}(x;\mathbf{a}+\mathbf{n}b)=\mathcal{Q}(x;\mathbf{n}b)\frac{\prod_{i=1}^u\prod_{k=1}^{\mathbf{e}_i\cdot\mathbf{a}}(1-x^{\mathbf{e}_i\cdot\mathbf{n}b+k})}{\prod_{j=1}^v\prod_{k=1}^{\mathbf{f}_j\cdot\mathbf{a}}(1-x^{\mathbf{f}_j\cdot\mathbf{n}b+k})}.
$$
If $\mathbf{a}/b$ is not in $\mathcal{D}$, then for each $\mathbf{t}$ in $e$ or $f$, no element of $\{1,\dots,\mathbf{t}\cdot\mathbf{a}\}$ is divisible by $b$. Hence, if $\xi_b$ is a complex primitive $b$th root of unity, then we have
$$
\mathcal{Q}(\xi_b;\mathbf{a}+\mathbf{n}b)=\mathcal{Q}(\xi_b;\mathbf{n}b)\frac{\prod_{i=1}^u\prod_{k=1}^{\mathbf{e}_i\cdot\mathbf{a}}(1-\xi_b^{k})}{\prod_{j=1}^v\prod_{k=1}^{\mathbf{f}_j\cdot\mathbf{a}}(1-\xi_b^{k})}=\mathcal{Q}(\xi_b;\mathbf{n}b)\mathcal{Q}(\xi_b;\mathbf{a}),
$$
so that
\begin{equation}\label{eq: inter1}
\mathcal{Q}(x;\mathbf{a}+\mathbf{n}b)\equiv\mathcal{Q}(x;\mathbf{n}b)\mathcal{Q}(x;\mathbf{a})\mod \phi_b(x)\mathbb{Z}[x].
\end{equation}
We shall prove that this congruence also holds when $\mathbf{a}/b$ belongs to $\mathcal{D}$. Indeed, in this case we have $\Delta(\mathbf{a}/b)\geq 1$ by assumption. By~\eqref{eq: poly}, the $\phi_b(x)$-valuation of $\mathcal{Q}(x;\mathbf{a}+\mathbf{n}b)$ is $\Delta(\frac{\mathbf{a}}{b}+\mathbf{n})=\Delta(\mathbf{a}/b)\geq 1$, and the $\phi_b(x)$-valuation of $\mathcal{Q}(x;\mathbf{a})$ is also $\Delta(\mathbf{a}/b)\geq 1$. Hence both polynomials are divisible by $\phi_b(x)$ and~\eqref{eq: inter1} holds.

Now we shall prove that
\begin{equation}\label{eq: inter2}
\mathcal{Q}(x;\mathbf{n}b)\equiv\mathcal{Q}(1;\mathbf{n})\mod \phi_b(x)\mathbb{Z}[x].
\end{equation}
We have
$$
\mathcal{Q}(x;\mathbf{n}b)=\frac{\prod_{i=1}^u\prod_{k=1}^{\mathbf{e}_i\cdot\mathbf{n}b}(1-x^k)}{\prod_{j=1}^v\prod_{k=1}^{\mathbf{f}_j\cdot\mathbf{n}b}(1-x^k)}=\frac{\prod_{i=1}^u\prod_{k=1}^{\mathbf{e}_i\cdot\mathbf{n}}(1-x^{kb})}{\prod_{j=1}^v\prod_{k=1}^{\mathbf{f}_j\cdot\mathbf{n}}(1-x^{kb})}\times\prod_{\ell=1}^{b-1}\frac{\prod_{i=1}^u\prod_{k=1}^{\mathbf{e}_i\cdot\mathbf{n}-1}(1-x^{\ell+kb})}{\prod_{j=1}^v\prod_{k=1}^{\mathbf{f}_j\cdot\mathbf{n}-1}(1-x^{\ell+kb})} \, \cdot
$$
From $|e|=|f|$, we also derive
$$
\frac{\prod_{i=1}^u\prod_{k=1}^{\mathbf{e}_i\cdot\mathbf{n}}(1-x^{kb})}{\prod_{j=1}^v\prod_{k=1}^{\mathbf{f}_j\cdot\mathbf{n}}(1-x^{kb})}=\frac{\prod_{i=1}^u\prod_{k=1}^{\mathbf{e}_i\cdot\mathbf{n}}\frac{1-x^{kb}}{1-x^b}}{\prod_{j=1}^v\prod_{k=1}^{\mathbf{f}_j\cdot\mathbf{n}}\frac{1-x^{kb}}{1-x^b}} 
\, ,
$$
which is a rational fraction without pole at $x=\xi_b$ and whose value at $\xi_b$ equals $\mathcal{Q}(1;\mathbf{n})$. Furthermore, for each $\ell$ in $\{1,\dots,b-1\}$, we have
$$
\frac{\prod_{i=1}^u\prod_{k=1}^{\mathbf{e}_i\cdot\mathbf{n}-1}(1-\xi_b^{\ell+kb})}{\prod_{j=1}^v\prod_{k=1}^{\mathbf{f}_j\cdot\mathbf{n}-1}(1-\xi_b^{\ell+kb})}=(1-\xi_b^\ell)^{(|e|-|f|)\cdot\mathbf{n}}=1.
$$
Since $\mathcal{Q}(x;\mathbf{n}b)\in\mathbb{Z}[x]$ and $\mathcal{Q}(\xi_b;\mathbf{n}b)=\mathcal{Q}(1;\mathbf{n})$, we obtain~\eqref{eq: inter2} as expected. 

\end{proof}

We now show how Proposition~\ref{prop: congq} yields on the one hand congruences through a specialization rule, and on the other hand algebraic independence results. 

Recall that, under the conditions of Proposition~\ref{prop: congq}, we have 
$$
F_{e,f}(q;\mathbf{x})=\sum_{\mathbf{n}\in\mathbb{N}^d}\mathcal{Q}_{e,f}(q;\mathbf{n})\mathbf{x}^{\mathbf{n}}\in \mathbb{Z}[q][[\mathbf{x}]]\,.
$$
Then the congruence relation in Proposition~\ref{prop: congq} is equivalent to 
$$
F_{e,f}(q;\mathbf{x})\equiv A(q;{\bf x})F_{e,f}(1;\mathbf{x}^b)\mod\phi_b(q)\mathbb{Z}[q][[{\bf x}]]
$$
for every positive integer $b$ and with the additional condition that $A(q;{\bf x})$ in $\mathbb{Z}[q][{\bf x}]$ satisfies $\deg_{x_i}A(q;{\bf x})\leq b-1$ for all $i$, $1\leq i\leq d$. The following proposition is the key to prove congruences for multisums of $q$-factorial ratios as in Proposition \ref{prop:indepqApery}.

\begin{prop}\label{prop:specialisationcongruence}
Assume the conditions of Proposition~\ref{prop: congq} are satisfied. Moreover, let $\bf{t}\in\mathbb{N}^d$ and ${\bf m}\in\mathbb{N}^d$ be such that  if $\mathbf{x}$ in $[0,1)^d$ satisfies ${\bf m}\cdot{\bf x}\geq1$, then $\Delta_{e,f}({\bf x})\geq1$. Then, for every positive integer $b$, we have: 
$$
F_{e,f}(q;q^{t_1}x^{m_1},\dots,q^{t_d}x^{m_d})\equiv B(q;x)F_{e,f}(1;x^{bm_1},\dots,x^{bm_d})\mod\phi_b(q)\mathbb{Z}[q][[{\bf x}]],
$$
where $B(q;x)$ is a one variable  polynomial in $\mathbb{Z}[q][x]$ satisfying $\deg_{x}B(q;x)\leq b-1$.
\end{prop}

Choosing $e=((2,1),(1,1))$ and $f=((1,0),(1,0),(1,0),(0,1),(0,1))$, we get that
$$
F_{e,f}(q;x,y)=\sum_{n_1,n_2\geq0}\frac{[2n_1+n_2]_q![n_1+n_2]_q!}{[n_1]_q!^3[n_2]_q!^2}x^{n_1}y^{n_2}\, .
$$ 
By Proposition 2 in \cite{DelaygueAp}, the function $\Delta_{e,f}$ is greater than or equal to $1$ on 
$\mathcal{D}_{e,f}$ so that the conditions of Proposition \ref{prop: congq} are satisfied. Furthermore, we can use Proposition~\ref{prop:specialisationcongruence} with $\mathbf{t}=(t,0)$ and $\mathbf{m}=(1,1)$ which yields
$$
F_{e,f}(q;q^{t}x,x)\equiv B(q;x)F_{e,f}(1;x^{b},x^{b})\mod\phi_b(q)\mathbb{Z}[q][[{\bf x}],
$$
where $B(q;x)$ is a polynomial in $\mathbb{Z}[q][x]$ satisfying $\deg_{x}B(q;x)\leq b-1$. A direct computation shows that
$$
F_{e,f}(q;q^{t}x,x)=\sum_{n=0}^\infty\sum_{k=0}^nq^{tk}\qbi{n}{k}{q}^2\qbi{n+k}{k}{q}x^n
$$
and
$$
F_{e,f}(1;x,x)=\sum_{n=0}^\infty\sum_{k=0}^n\binom{n}{k}^2\binom{n+k}{k}x^n \,.
$$
This yields the congruences for $q$-analogs of the first Ap\'ery sequence $a_n(q)$ 
given in Proposition~\ref{prop:qApery}. The result for the second Ap\'ery sequence $b_n(q)$ is derived along the same line.

To prove Proposition~\ref{prop:indepqApery}, we first show by Proposition \ref{prop:specialisationcongruence} that each series $f(q;x)$ in $\mathcal{F}_q$ belongs to $\mathcal{D}(q;f(1;x))$. For example, we use the following specialization associated with $\mathbf{t}=(0,0)$ and $\mathbf{m}=(1,1)$: 
$$
\sum_{n=0}^\infty\sum_{k=0}^n\qbi{n}{k}{q}^rx^n=F_{e,f}(q;x,x)\, ,
$$   
where
$$
F_{e,f}(q;x,y)=\sum_{n_1,n_2\geq 0}\frac{[n_1+n_2]_q!^r}{[n_1]_q!^r[n_2]_q!^r}x^{n_1}y^{n_2}\,.
$$
By Proposition 1.2 and Section 9.3 in \cite{ABD1}, we know that $\mathcal{F}_1$ (the set of all series $f(1;x)$) is a subset of $\mathfrak{L}_1$ and that all elements of $\mathcal{F}_1$ are algebraically independent over $\mathbb{C}(x)$. Hence Theorem \ref{thm: propagation} implies that, for every nonzero complex number $q$, all elements of $\mathcal{F}_q$ are algebraically independent over $\mathbb{C}(x)$.

\section{Sketch of proof of Theorem \ref{thm: propagation} }\label{sec:proofpropagation}

Though Theorem \ref{thm: propagation} holds true for all nonzero complex number $q$, 
we will focus here on the case where $q$ is an algebraic number.   
The case where $q$ is transcendental is actually simpler even if it requires specific 
considerations we do not want to deal with here for space limitation.  

Throughout this section, we fix a nonzero algebraic number $q$. 
We let $K$ be the number field $\mathbb Q(q)$ and  
 $R:=\mathcal O(K)$ be its ring of integers. Recall that $R$ is thus a Dedekind domain. 

The proof of Theorem \ref{thm: propagation}  relies on  the following Kolchin-like proposition 
which is a special instance of Proposition~$4.3$ in \cite{ABD1}.

\begin{prop}\label{prop: Kolchin}
Let $p$ be a prime number, $F$ be a finite extension of degree $d_{p}$ of $\mathbb F_p$, 
and $k$ be a positive integer such that $d_p\mid k$.  
Let $f_1,\ldots,f_n$ be nonzero power series in $F[[{\bf x}]]$ satisfying 
$f_i({\bf x})= A_i({\bf x})f_i({\bf x}^{p^k})$ for some $A_i\in F[{\bf x}]$ and every $1\leq i\leq n$. If $f_1,\dots,f_n$ satisfy 
a nontrivial polynomial relation of degree $d$ with coefficients in $F({\bf x})$, 
 then there exist 
$m_1,\dots,m_n\in\mathbb{Z}$, not all zero, and a nonzero $r({\bf x})\in F({\bf x})$ such that
$$
A_1({\bf x})^{m_1}\cdots A_n({\bf x})^{m_n}=r({\bf x})^{p^k-1} \, .
$$ 
Furthermore,  $|m_1+\cdots+m_n|\leq d$ and $|m_i|\leq d$ for $1\leq i\leq n$.
\end{prop}

We will also need the following result which will enable us to connect 
reductions modulo prime numbers and modulo cyclotomic polynomials.

\begin{prop}\label{prop: cycloprime}
There exists an infinite set $\mathcal S$ of maximal ideals of $R$ such that, 
for all $\mathfrak p\in \mathcal S$, we have 
$\mathbb Z[q]\subset R_{\mathfrak p}$ and $\phi_{b}(q)\mathbb Z[q]\subset 
\mathfrak{p}R_{\mathfrak p}$ for some 
prime number $b$ (depending on $\mathfrak p$).
\end{prop}

For space limitation, Proposition \ref{prop: cycloprime} will not be proved here. 
Its proof is elementary when $q$ is a root of unity and relies on the $S$-unit theorem 
otherwise.  
We will also need the  two following auxiliary results,  the first of which being  
Lemma~$4.4$ in~\cite{ABD1}.

\begin{lem}
\label{lem: 0} Let $R$ be a Dedekind domain, $K$ be its field of fractions, and $g_1,\ldots, g_n$ be
power series in $R[[{\bf x}]]$.  If 
$g_{1 \vert \mathfrak p},\ldots, g_{n \vert \mathfrak p}$ are linearly dependent over 
$R/\mathfrak p$ for infinitely many maximal ideals $\mathfrak p$,   
then $f_1,\ldots, f_n$ are linearly dependent over $K$. 
\end{lem}

\begin{lem}\label{lem: b}
Let $K$ be a commutative field and set $b$ a positive integer. Let $r(\mathbf{x})\in K(\mathbf{x})$ and $s(\mathbf{x})\in K(\mathbf{x})\cap K[[\mathbf{x}]]$ be two rational fractions such that $s(\mathbf{0})\neq0$. If there exists a nonzero (mod $p$ if $char(K)=p$) integer  $m$ satisfying $s(\mathbf{x}^b)=r(\mathbf{x})^m$, then there exists $t(\mathbf{x})$ in $K(\mathbf{x})$ such that $r(\mathbf{x})=t(\mathbf{x}^b)$. 
\end{lem}

\begin{proof}[Proof of Theorem \ref{thm: propagation}] 
Let $\mathcal S$ be the set of maximal ideals of $R:=\mathcal O(K)$ given in Proposition 
\ref{prop: cycloprime}. 
With all $\mathfrak p$ in $\mathcal S$, we associate a prime number $p$ such that the residue field 
$R/\mathfrak p$ is a finite field of characteristic $p$ so that $p\mathbb Z\subset \mathfrak p$. 
Let $d_{\mathfrak p}$ be the degree of the field extension 
$R/\mathfrak p$ over $\mathbb F_p$.  
Since $g_i$ belongs to $\mathcal L_d$, there exists a polynomial $A_{i}\in \mathbb Z[{\bf x}]$ such that 
$$
g_i({\bf x}) \equiv A_{i}({\bf x}) g_i({\bf x}^{p^{k_i}}) \mod \mathfrak p[[\mathbf{x}]]
$$
with $\deg_{x_j} A_{i}\leq p^{k_{i}}-1$. 
We set  $k:=\mathrm{lcm}(d_{\mathfrak p},k_{1},\dots,k_{n})$. 
Then iterating the above relation, for all $i$ in $\{1,\dots,n\}$ and all $\mathfrak p$ in $\mathcal S$, there exists 
$B_{i}(\mathbf{x})$ in $\mathbb Z[{\bf x}]$ satisfying
\begin{equation}\label{eq: Bi}
g_i(\mathbf{x})\equiv B_{i}(\mathbf{x})g_i\big(\mathbf{x}^{p^{k}}\big)\mod \mathfrak p[[\mathbf{x}]],
\end{equation}
with $\deg_{x_j}(B_{i})\leq  p^{k}-1$. 

Now let us assume by contradiction that $f_1,\dots,f_n$ are algebraically dependent over 
 $\mathbb{C}(\mathbf{x})$ and thus over $K(\mathbf{x})$ for the coefficients of the 
 formal power series $f_i$ belong to $K$ (see for instance \cite{ABD1}). 
 Let $Q(\mathbf{x},y_1,\dots,y_n)$ be a nonzero polynomial in 
 $R[\mathbf{x}][y_1,\dots,y_n]$ of total degree at most 
$\kappa$ in $y_1,\dots,y_n$ such that  $
Q\big(\mathbf{x},f_1(q;\mathbf{x}),\dots,f_n(q;\mathbf{x})\big)=0. 
$ 
Since  $f_i\in\mathcal{D}(q;g_i)$, for every $i$ in $\{1,\dots,n\}$,  
Proposition \ref{prop: cycloprime} implies that 
$
f_i(q;\mathbf x)\equiv A_{i}(q;{\bf x}) g_i(\mathbf{x}^{b}) \mod 
\mathfrak pR_{\mathfrak p}[[\mathbf{x}]]$, for some prime $b$.  
Since $Q$ and the series $f_i$ are all nonzero and $R$ is a Dedekind domain, there thus exists 
an infinite subset $\mathcal S'$ of $\mathcal S$ such 
that, for every $\mathfrak p$ in $S'$,  the relation
$$
Q\big(\mathbf{x}, A_{1}(q;{\bf x}) g_1(\mathbf{x}^{b}),\dots, A_{n}(q;{\bf x}) 
g_n(\mathbf{x}^{b})\big)\equiv 0\mod \mathfrak pR_{\mathfrak p}[[\mathbf{x}]]\,$$
provides a nontrivial algebraic relation over $R_{\mathfrak p}/\mathfrak p =R/\mathfrak p$ 
between the series $g_{i\vert {\mathfrak p}}(\mathbf{x}^b)$.  
By \eqref{eq: Bi}, one has 
$g_i(\mathbf{x}^{b})\equiv B_{i}(\mathbf{x}^{b})g_i\big(\mathbf{x}^{bp^{k}}\big)
\mod \mathfrak{p}[[\mathbf{x}]]$ and Proposition~\ref{prop: Kolchin} 
then applies to $g_{1|\mathfrak{p}}(\mathbf{x}^{b}),\dots,g_{n|\mathfrak{p}}(\mathbf{x}^{b})$ by taking 
$F=R/\mathfrak p$.  
 There thus exist integers $m_{1},\dots,m_{n}$, not all zero, 
and a nonzero rational fraction $r(\mathbf{x})$ in $F(\mathbf{x})$ such that
\begin{equation}\label{eq: Bi0}
B_{1|{\mathfrak p}}(\mathbf{x}^{b})^{m_{1}}\cdots B_{n| \mathfrak p}(\mathbf{x}^{b})^{m_{n}}
=r(\mathbf{x})^{p^k-1}.
\end{equation}
As $g_i$ belongs to $\mathcal{L}_d$, the constant coefficient in the left-hand side of \eqref{eq: Bi0} is equal to $1$.  By Lemma~\ref{lem: b}, as $p^k-1\neq0$ mod $p$, there exists a rational fraction $u(\mathbf{x})$ in $F(\mathbf{x})$ such that $r(\mathbf{x})=u(\mathbf{x}^{b})$ 
and we obtain that
 $B_{1| \mathfrak p}(\mathbf{x})^{m_{1}}\cdots B_{n| \mathfrak p}(\mathbf{x})^{m_{n}}=
u(\mathbf{x})^{p^k-1}$. Furthermore, we have $|m_1+\cdots+m_n|\leq \kappa$ and 
$|m_i|\leq \kappa$ for $1\leq i\leq n$.
Note that the rational fractions $B_i$, $u$ and the integers $m_i$ all depend on $\mathfrak{p}$. 
However, since all the integers $m_i$ belong to a finite set, the pigeonhole principle implies the existence 
of  an infinite subset $\mathcal{S}''$ of $\mathcal{S}'$ and of integers $t_1,\dots,t_n$
such that, for all $\mathfrak p$ in $\mathcal{S}''$, we have $m_i=t_i$ for $1\leq i \leq n$.  
We can thus assume      that $\mathfrak p$  belongs to $\mathcal{S}''$   
and write $u(\mathbf{x})=s(\mathbf{x})/t(\mathbf{x})$ with $s(\mathbf{x})$ and $t(\mathbf{x})$ 
in $F[{\bf x}]$ and coprime.  
Since $\deg B_{i}\leq  p^k-1$, the degrees of $s({\bf x})$ and $t({\bf x})$ are bounded by 
$|t_1|+\cdots+|t_n|\leq n\kappa $.
Set $h(\mathbf{x}):=g_1(\mathbf{x})^{-t_1}\cdots g_n(\mathbf{x})^{-t_n}\in  \mathbb{Z}[[\mathbf{x}]]
\subset R[[{\bf x}]]$. Then we obtain that 
\begin{align*}
h_{| \mathfrak{p}}\big(\mathbf{x}^{p^k}\big)&= g_{1|\mathfrak{p}}\big(\mathbf{x}^{p^k}\big)^{-t_1}
\cdots g_{n|\mathfrak{p}}\big(\mathbf{x}^{p^k}\big)^{-t_n}\\
&=g_{1|\mathfrak{p}}(\mathbf{x})^{-t_1}\cdots g_{n|\mathfrak{p}}(\mathbf{x})^{-t_n}B_{1|\mathfrak{p}} (\mathbf{x})^{t_1}\cdots B_{n|\mathfrak{p}}(\mathbf{x})^{t_n}\\
&=h_{|\mathfrak{p}}(\mathbf{x})u(\mathbf{x})^{p^k-1}.
\end{align*}
Since $h_{|\mathfrak p}$ is nonzero, we obtain that 
$h|_\mathfrak{p}(\mathbf{x})^{p^k-1}=u(\mathbf{x})^{p^k-1}$ and there exists $a$ 
in a suitable algebraic extension of 
$F$  
such that $h_{|\mathfrak p}(\mathbf{x})=au(\mathbf{x})$.  As the coefficients of 
$h_{|\mathfrak p}$ and $u$ belong to $R/\mathfrak{p}$, we get $a\in R/\mathfrak{p}$.   
Thus for infinitely many maximal ideals $\mathfrak p$, the reduction modulo 
$\mathfrak p$ of the power series $x_i^mh({\bf x})$ and $x_i^m$, $1\leq i\leq n$, 
$0\leq m\leq n\kappa$, are linearly dependent over $R/\mathfrak p$. Since $R$ is a Dedekind domain, 
Lemma~\ref{lem: 0}  
implies that these power series are linearly dependent over $K$, which means that 
$h(\mathbf{x})$ belongs to $K({\bf x})$. This is a contradiction as $g_1,\dots,g_n$ are algebraically independent over $\mathbb{C}(\mathbf{x})$.
\end{proof}






\medskip

\begin{small}
\address{B. Adamczewski, \'E. Delaygue and F. Jouhet, Univ Lyon, Universit\'e Claude Bernard Lyon 1, CNRS UMR 5208, Institut Camille Jordan, F-69622 Villeurbanne Cedex, France. Emails: boris.adamczewski@math.cnrs.fr, delaygue@math.univ-lyon1.fr and jouhet@math.univ-lyon1.fr}
\medskip

\address{J. P. Bell, Department of Pure Mathematics, University of Waterloo, Waterloo, ON, Canada. Email: jpbell@uwaterloo.ca}

\end{small}


\begin{thebibliography}{1}

\bibitem{AB13} \textsc{B. Adamczewski and J. P. Bell}, {\it Diagonalization and rationalization of algebraic Laurent series}, 
Ann. Sci. \'Ec. Norm. Sup\'er. {\bf 46}  (2013), 963--1004. 

\bibitem{ABD1}\textsc{B. Adamczewski, J. P. Bell, and E. Delaygue}, Algebraic independence of $G$-functions and congruences ``\`a la Lucas", preprint (2016), arXiv:1603.04187, 50 pages.


 
\bibitem{BH} \textsc{F. Beukers and G. Heckman}, \textit{Monodromy for the hypergeometric functions ${}_{n}F_{n-1}$}, Invent. Math. {\bf 95} (1989), 325--354. 

\bibitem{BM06} \textsc{M. Bousquet-M\'elou}, \textit{Rational and algebraic series in combinatorial enumeration}, International Congress of Mathematicians. Vol. III, 789--826, 
Eur. Math. Soc., Z\"urich, 2006. 




\bibitem{DelaygueAp}\textsc{E. Delaygue}, \textit{Arithmetic properties of Ap\'ery-like numbers}, preprint (2013), arXiv:1310.4131v2, 25 pages. 


\bibitem{Fl} \textsc{P. Flajolet}, \textit{Analytic Models and Ambiguity of Context-Free Languages},  Theor. Comput. Sci. {\bf 49} (1987),  283--309. 

\bibitem{Fr} \textsc{R. D. Fray}, \textit{Congruence properties of ordinary and $q$-binomial coefficients}, Duke Math. J. {\bf 34} (1967), 467--480.



\bibitem{GZ}\textsc{V. J. W. Guo and J. Zeng}, \textit{Some congruences involving central $q$-binomial coefficients}, Adv. Appl. Math. {\bf 45} (2010), 303--316.

\bibitem{KW}\textsc{D. E. Knuth and H. S. Wilf}, \textit{The power of a prime that divides a generalized binomial coefficient}, J. Reine Angew. Math. {\bf 396} (1989), 212--219.


\bibitem{Lu}\textsc{E. Lucas}, \textit{Sur les congruences des nombres eul\'eriens et les coefficients diff\'erentiels des fonctions trigonom\'etriques, suivant un module premier}, 
Bull. Soc. Math. France {\bf 6} (1877--1878), 49--54.  

\bibitem{Pa}\textsc{H. Pan}, \textit{A Lucas-type congruence for $q$-Delannoy numbers}, 11 pages, preprint arXiv:1508.02046, 2015.

\bibitem{Sa}\textsc{B. E. Sagan}, \textit{Congruence Properties of $q$-Analogs}, Adv. Math. {\bf 95}, 127--143 (1992).


\bibitem{SW2} \textsc{H. Sharif and C. F. Woodcock}, 
\textit{On the transcendence of certain series}, 
 J. Algebra {\bf 121} (1989), 364--369. 

\bibitem{St80} \textsc{R. P. Stanley}, \textit{Differentiably finite power series},  
European J. Combin. {\bf 1} (1980), 175--188.

\bibitem{St} \textsc{V. Strehl}, \textit{Zum $q$-Analogon der Kongruenz von LUCAS}, in ``S\'eminaire Lotharingien de Combinatoire, 5-i\`eme Session'' (J. D\'esarmenien, Ed.), pp. 102--104, Institut de Recherche 
Math\'ematique Avanc\'ee, Strasbourg, 1982.

\bibitem{WZ} \textsc{S.O. Warnaar and W. Zudilin }, \textit{A $q$-rious positivity}, Aeq. Math. {\bf 81} (2011), 177--183.


\end{thebibliography}
\end{document}